\documentclass[11pt]{amsart}
\usepackage{enumitem}
\usepackage{color}
\usepackage{soul}
\numberwithin{equation}{section}

\newtheorem{theorem}{Theorem}[section]

\newtheorem{proposition}[theorem]{Proposition}
\newtheorem{lemma}[theorem]{Lemma}
\newtheorem{corollary}[theorem]{Corollary}
\newtheorem{remark}[theorem]{Remark}

\def\al{\aligned}
\def\eal{\endaligned}
\def\be{\begin{equation}}
\def\ee{\end{equation}}
\def\lab{\label}
\def\a{\alpha}

\def\M{{\bf M}}
\def\g{{\bar}}
\def\lam{{\lambda}}

\def\al{\aligned}

\def\g{\bar}
\def\pa{\partial}
\def\d{\nabla}

\DeclareMathOperator{\Ric}{Ric}

\numberwithin{equation}{section}

\begin{document}

\title[]{Gradient and Eigenvalue Estimates on the canonical bundle of K\"ahler manifolds}
\author{Zhiqin Lu, Qi S. Zhang and Meng Zhu}
\address{Department of Mathematics, University of California, Irvine, CA 92697, USA}
\email{zlu@math.uci.edu}
\address{Department of Mathematics, University of California, Riverside, CA 92521, USA}
\email{qizhang@math.ucr.edu}
\address{School of Mathematical Sciences and Shanghai Key Laboratory of PMMP, East China Normal University, Shanghai 200241, China}
\email{mzhu@math.ecnu.edu.cn}
\date{08/2020; MSC 2020: 58A10, 58J35, 58J50}

\begin{abstract}
We prove certain gradient and eigenvalue estimates, as well as the heat kernel estimates, for the Hodge Laplacian on $(m,0)$ forms, i.e., sections of the canonical bundle of K\"ahler manifolds, where $m$ is the complex dimension of the manifold. Instead of the usual dependence on curvature tensor, our condition depends only on the Ricci curvature bound. The proof is based on a new Bochner type formula for the gradient of $(m, 0)$ forms, which involves only the Ricci curvature and the gradient of the scalar curvature.

\end{abstract}
\maketitle
\tableofcontents

\section{Introduction}

In this paper, we prove, under conditions on the Ricci curvature alone, certain gradient and eigenvalue estimates for solutions of some elliptic and parabolic equations involving the Hodge Laplacian on sections of the canonical bundle of K\"ahler manifolds. Let us recall that a section of the canonical bundle of a K\"ahler manifold with complex
dimension $m$ is an $(m, 0)$ form. Since the appearance of De Rham and Hodge theory, there has been a vast literature on the study of Hodge Laplacian acting on differential forms due to its significance to analysis, geometry and topology.
 See for example the papers \cite{Bo, Ko, KW, Li:1, Do, CT, Lo, Ma, WZ, CL}.  The general paradigm for gradient and eigenvalue estimates on $p$ forms is that results concerning $0$ forms e.g. scalar functions involve the Ricci curvature and
results for one and higher forms involve the curvature tensor in general. The reason is that the Bochner--Weitzenb\"ock formula for $p$ forms with $p \ge 1$ involve the full curvature tensor.  One exception is for spin bundles where the Lichnerowicz formula involves only the scalar curvature. In this paper we manage to find another exception: $(m,0)$ forms on K\"ahler manifolds. In this case the Bochner--Weitzenb\"ock formula  for gradients involves only the Ricci curvature and the gradient of the scalar curvature. One implication of this result is that gradient estimates on  heat kernels and eigenvalue estimates for $(m,0)$ forms on Calabi-Yau manifolds are free of curvature conditions.
We mention that for holomorphic sections in the canonical bundle, it is well known that the Bochner formula for gradient of the sections depends only on the Ricci curvature. See \cite{Ti}  (5.15) e.g. This fact plays an important role in complex differential geometry, especially on the parts involving the Bergman kernel.  What we did here is removing the holomorphic condition.

The main results of the paper are Theorem \ref{thmain} below and Lemma \ref{Bochner} in the next section.
In the theorem, we obtain a qualitatively sharp lower bound for all eigenvalues of the Hodge Laplacian on the canonical bundle of K\"ahler manifolds. The curvature conditions are bounded Ricci curvature and nonnegative scalar curvature which is positive somewhere. Similar lower bounds for eigenvalues higher than certain Betti numbers and under bounded curvature assumption are proven in  J.P. Wang-L.F. Zhou \cite{WZ} for general compact manifolds and forms. Finding a lower bound for  eigenvalues has been an active research topic for many years.  We mention that although
an $(m, 0)$ form looks like a scalar function, it behaves more like a vector field since the coefficient are complex valued. For this reason we are unable to extend the classical method of integrating a gradient bound in Li-Yau \cite{LY} to bound the first eigenvalue from below, which is the first step in bounding other eigenvalues from below.

Before stating the main results, we need to set up the conventions first. Let $({\M}, g_{i\bar{j}})$ be a compact K\"ahler manifold of complex dimension $m$. Denote by $\Delta_d$ and $\Delta_{\bar{\pa}}$ the real and complex Hodge Laplacian, respectively. Denote by $\Delta=-\d^*\d$ the complex rough Laplacian, which is $\d_i\d_{\bar{i}}$ in local normal coordinates, then $\Delta=-\Delta_{\bar{\pa}}=-\frac{1}{2}\Delta_d$ when they act on functions and $(p,0)$ forms (see e.g. Theorem 6.1 on p119 in \cite{MK}). Let $Ric$ stand for the Ricci curvature tensor, and $R_{i\bar{j}}$ is the component of $Ric$ in holomorphic coordinates. Let $R$ be the scalar curvature. The volume of $\M$ is denoted by $|\M|$, and the diameter of $\M$ is denoted by $diam(\M)$. Let $K_\M=\Lambda^m(T^*\M)$ be the canonical bundle of $\M$, and $L^2(\M, K_\M)$ the space of all smooth sections of $K_\M$, i.e., $(m,0)$ forms on $\M$.

\begin{theorem}
\lab{thmain}
Let $(\M^m, g_{i\bar{j}})$ be a compact K\"ahler manifold of complex dimension $m\geq 2$. Denote by $\lambda_k$ the $k$-th eigenvalue of the Hodge Laplacian acting on $L^2(\M, K_\M)$, and $\lambda^0_1$ the first  positive eigenvalue of the scalar Laplacian. Suppose that the Ricci curvature satisfies $|Ric|\leq K$ and the scalar curvature satisfies:  $R\geq 0$ and is positive somewhere. Then the following lower bounds are true.

\be
\al
\lambda_1 &\ge \min\{\frac{1}{2} \lambda^0_1 + \inf_\M R, \quad \frac{\lambda^0_1}{2(\lambda^0_1 + \sup R)} \frac{1}{|\M|} \int_\M R\} \equiv c_0>0;\\
\lambda_k &\geq c_1 k^{\frac{1}{m}}, \qquad  \forall\ k\geq 2,
\eal
\ee
where $c_1=\min\{c_0, \left(\frac{m}{\Lambda(m+1)}\right)^{\frac{1}{m}}\}$, $\Lambda=2(2m)^{m+1}16^{2m(m-1)}C_S^m|\M|\left(1+\frac{K+|\M|^{-\frac{1}{m}}}{c_0}\right)^{m+1}$, and $C_S$ is the $L^2$ scalar Sobolev constant, namely, the smallest constant such that
\be
\lab{sob}
\left(\int_{\M}|f|^{\frac{2m}{m-1}}\right)^{\frac{m-1}{m}}\leq C_S\left(\int_{\M}|\d f|^2+|\M|^{-\frac{1}{m}}\int_{\M}|f|^2\right)
\ee
for any smooth function $f$ on $\M$.
\end{theorem}

\begin{remark}
For Riemann surfaces, i.e., the case $m=1$, the picture of $\lambda_1$ is clearer. We may start from the Euler characteristic $\chi(\M)=b_0-b_1+b_2=2b_0-b_1=2-b_1=2-2g$, where $b_k$ is the kth Betti number, and $g$ is the genus of $\M$. Thus, it follows that $b_1=g$. So when $g\geq 1$, we have $\lambda_1=0$ since $2h^{1,0}=b_1>0$, where $h^{1,0}$ is the $(1,0)$ Hodge number.

When $g=0$, i.e., $b_1=0$, one concludes that $\lambda_1>0$. Let a (1,0) form $\phi$ be an eigenform of $\lambda_1$. Since $\partial \phi=0$ and $\lambda_1>0$, it follows that $\partial^* \phi\neq 0$ and $\Delta_d\partial^*\phi=\partial^*\Delta_d\phi=\lambda_1\partial^*\phi$. In other words, $\partial^*\phi$ is an eigenfunction of $\Delta_d$. Thus, one gets $\lambda_1\geq \lambda_1^0$.
\end{remark}

\begin{remark}
It is well known that
\[
\int_{\M}R\, \omega^n=m\pi\int_{\M} \mathbf{c}_1\wedge[\omega]^{m-1},
\]
which is independent of the choice of the K\"ahler metric in the K\"ahler class. Here $\mathbf{c}_1$ is the first Chern class of $\M$.
\end{remark}

\begin{remark} In general, it is impossible to find a positive lower bound for $\lambda_1$, the first eigenvalue.  Indeed,  $\lambda_1$ can be $0$ if the scalar curvature is $0$. It is a standard fact that
the Hodge number  $h^{m,0}$  for Calabi-Yau m-manifolds is $1$, i.e. the space of harmonic $(m, 0)$ forms
is one dimensional, and therefore $\lambda_1=0$.

On the other hand, using the fact that the Bochner formula for all $(p, 0)$ forms involve only the Ricci curvature, one can prove an upper bound for the heat kernel of the Hodge Laplacian on $(p, 0)$ forms which depends only on volume of $\M$ and the bound on Ricci curvature. See for example Theorem \ref{thhk} below.
Then it is  well known (c.f. section 5 in \cite{LY:2}) that the upper bound  for the trace of the heat kernel implies
\[
\lambda_j \ge C_1 \, j^{1/m}
\] for  $j \ge j_0>0$. Here $j_0$ and $C_1$ depends only on $m, p,  |\M|$ and $K$, the Ricci bound.
\end{remark}

\begin{remark}
In the theorem above, the parameters $c_0$ and $C_S$ can be explicitly estimated in terms of geometric quantities. Following Li-Yau,  (see \cite{LY} and Li' book \cite{Li} Theorem 5.7, and Yang \cite{Ya}), There are dimensional constant $C_2$ such that
\[
\lambda^0_1 \ge \frac{ \pi^2}{ diam(\M)^2} \exp(-C_2\, diam(\M) \sqrt{K}).
\]As well known by the work of Zhong-Yang \cite{ZY}, if $Ric \ge 0$, then a sharp lower bound holds: $\lambda^0_1 \ge \frac{\pi^2}{ diam(\M)^2}$.

 The Sobolev constant $C_S$ depends on $m$, lower bound $K$ of the Ricci curvature, diameter upper bound $D$ and volume lower bound $V$ of $\M$ (see e.g., \cite{Cro}, \cite{Yau}). Indeed, by a combination of the results in \cite{Li} and \cite{Cro}, the following explicit upper bound of $C_S$ and hence a lower bound of $\lambda_k$ can be established.
\be
\lab{sobcst}
C_S\leq 2^{\frac{m-1}{m(2m-1)}}\max\left\{\left(\frac{m-1}{2m-1} IN_{\alpha}(\M)\right)^{-2},2^{\frac{2m-3}{m-1}}\right\},
\ee with
\be
\lab{isoal}
IN_{\alpha}(\M)\geq \left(\frac{1}{4\omega_n^{n-1}\omega_{n-1}}\right)^{\frac{1}{n}}\left(\frac{|\M|}{\int_0^D [\sqrt{K^{-1}}\sinh(\sqrt{K}r)]^{n-1}dr}\right)^{\frac{n+1}{n}}.
\ee Here $n=2m$ is the real dimension of $\M$, and $\omega_n$ denotes the volume of the unit $n$-sphere.

The proof goes as follows.
Let $\alpha=\frac{n}{n-1}$, and $\beta=\frac{2\alpha-2}{2-\alpha}=\frac{2}{n-2}$.
Using the inequality
\[
1-x^{\beta}\leq (1-x)^{\beta}\leq 2^{1-\beta}-x^{\beta},\ 0\leq x, \beta\leq 1,
\]
one can see that in the proof of Corollary 9.9 in \cite{Li}, one can take $C_3=2^{1-\beta}=2^{\frac{n-4}{n-2}}$, $C_4=1$, $C_5=C_6=2C_3$, and the inequality (9.8) therein becomes
\be\label{NSobolev}
\int_{\M}|\d f|^2\geq \left(\frac{n-2}{2n-2} SN_{\alpha}(\M)\right)^2\left[2^{-\frac{n-2}{n(n-1)}}
\left(\int_{\M}|f|^{\frac{2n}{n-2}}\right)^{\frac{n-2}{n}}-C_5|\M|^{-\frac{2}{n}}\int_{\M}|f|^2,\right].
\ee Furthermore, according to Theorem 9.6 in \cite{Li} and Theorem 13 in \cite{Cro}, the constant $SN_{\alpha}(\M)$ satisfies
the following estimates:
\[
IN_{\alpha}(\M)\leq SN_{\alpha}(\M)\leq 2^{1/n}IN_{\alpha}(\M),
\] where the Neumann isoperimetric constant $IN_{\alpha}$ satisfies (\ref{isoal}).
Combining the estimates above and replacing $n$ by $2m$, one can get the upper bound of $C_S$ in (\ref{sobcst}).

\end{remark}

 The rest of the paper is organized as follows.
In Section 2, we will state and prove the Bochner type formula for the gradient of $(m, 0)$ forms.
It is well known, c.f. \cite{MK}, that the Bochner formula for all $(p, 0)$ forms with $p=1, 2, \cdots , m$
involve only the Ricci curvature.  As mentioned, what we will prove here is that for the gradient of $(m, 0)$ forms, the Bochner formula involve only the Ricci curvature and gradient of the scalar curvature.
With this formula in hand, we will roughly follow the steps in \cite{Li:1} and \cite{WZ} to prove Theorem
\ref{thmain}. One difference is that we need to find a reasonable condition such that the first eigenvalue $\lambda_1$ has a positive lower bound, which as mentioned could be $0$. It turns out that this is related to the first positive eigenvalue of the scalar Laplacian and the total scalar curvature and hence the K\"ahler class. Then we need to prove a mean value inequality for the eigenfunctions and their gradients. These will be used in an algebraic iteration process to prove the lower bound.

In Section 3, we extend some gradient estimates to the corresponding heat equation.  Just like the Hodge Laplacian, heat equation on $(m, 0)$ forms is of interest to several areas.  See for example a recent paper \cite{MMZ}.



\section{New Bochner formula and elliptic estimates}

In this section we will first present the Bochner formula as the main lemma.  Then Theorem \ref{thmain} will be proven at the end of the section after a number of intermediate results stated as lemmas.

We will need of the following basic formulas.

Denote by $\{z_1, z_2,\cdots, z_m\}$ a local holomorphic coordinate system. Let the Hermitian metric be \[\displaystyle h= \sum_{i,j=1,2,\cdots,m}g_{i\bar{j}}dz^i \otimes d\bar{z}^{j},\] where the underlying Remannian metric is given by $g=2 Re(h)$. The Christoffel symbols are given by
\[
\Gamma^{i}_{jk}=g^{i\bar{l}}\frac{\pa g_{j\bar{l}}}{\pa z_k}=g^{i\bar{l}}\frac{\pa g_{k\bar{l}}}{\pa z_j}.
\]
The curvature and Ricci curvature tensors are defined by
\be\label{curvature}
[\d_{\frac{\partial }{\partial z_i}},\d_{\frac{\partial }{\partial \g z_j}}]\frac{\pa}{\pa z_k}=-\frac{\pa \Gamma^l_{ik}}{\pa \bar{z}_j}\frac{\pa}{\pa z_l}=R_{i\bar{j}k}^l\frac{\pa}{\pa z_l},\qquad [\d_{\frac{\partial }{\partial z_i}},\d_{\frac{\partial }{\partial \g z_j}}]dz^l=\frac{\pa \Gamma^l_{ik}}{\pa \bar{z}_j}dz^k=-R_{i\bar{j}k}^ldz^k,
\ee
\be
g_{s\bar{l}}R^s_{i\bar{j}k}=R_{i\bar{j}k\bar{l}}, \ g^{l\bar{s}}R_{i\bar{j}k\bar{s}}=R^l_{i\bar{j}k},
\ee
and
\be\label{Ricci curvature}
R_{i\bar{j}}=g^{k\bar{l}}R_{i\bar{j}k\bar{l}}=R^{k}_{i\bar{j}k}=-\partial_i\partial_{\g j} \log\det(g).
\ee
Here and in the following, Einstein summation convention is used, namely, repeated indices are implicitly summed over. Readers can also refer to the book of Morrow-Kodaira \cite{MK} Chapter 3 for more details and formulas for K\"ahler geometry. Notice that the curvature tensor in \cite{MK} differs from the one defined in \eqref{curvature} by a negative sign.

Thus, for any $(p,q)$ form
\be\label{def form}
\phi=\sum_{\substack{i_1<i_2\cdots<i_p\\ j_1<j_2<\cdots<j_q}}\phi_{i_1\cdots i_p\bar{j}_1\cdots \bar{j}_q}dz^{i_1}\wedge \cdots \wedge dz^{i_p}\wedge d\bar{z}^{j_1}\wedge\cdots \wedge d\bar{z}^{j_q},
\ee
the following Ricci identity holds in local normal coordinates.
\be\label{Ricci identity}
\al
&(\d_a\d_{\bar{b}}-\d_{\bar{b}}\d_a)\phi\\
=&-\sum_{t=1}^p R^l_{a\bar{b}i_t}\phi_{i_1\cdots i_{t-1}li_{t+1}\cdots i_p \bar{j}_1\cdots \bar{j}_{q}}-\sum_{t=1}^q R^{\bar{l}}_{a\bar{b}\bar{j}_t}\phi_{i_1\cdots i_p \bar{j}_1\cdots \bar{j}_{t-1}\bar{l}\ \bar{j}_{t+1}\cdots\bar{j}_{q}}\\
=&-\sum_{t=1}^p R_{a\bar{b}i_t\bar{l}}\phi_{i_1\cdots i_{t-1}li_{t+1}\cdots i_p \bar{j}_1\cdots \bar{j}_{q}}-\sum_{t=1}^q R_{a\bar{b}\bar{j}_tl}\phi_{i_1\cdots i_p \bar{j}_1\cdots \bar{j}_{t-1}\bar{l}\ \bar{j}_{t+1}\cdots\bar{j}_{q}}.
\eal
\ee Similar identities also hold for $p$ covariant and $q$ contravariant tensors.
The K\"ahler form and Ricci form are
\be\label{Ricci curvature}
\omega=\frac{\sqrt{-1}}{2}g_{i\bar{j}}dz^i\wedge d{\g z}^{j}\ , \textrm{and} \ \rho=\frac{\sqrt{-1}}{2}R_{i\bar{j}}dz^i\wedge d{\g z}^j,
\ee
respectively.

Also, during the course of proofs below, there are situations in which complex gradient and Laplacian on real functions need to be transferred into the real ones, e.g. using min-max definition to get the first nonzero eigenvalue of the Laplacian on functions. For this purpose, be aware that
$$\frac{\partial }{\partial z_i}=\frac{1}{2}\left(\frac{\partial }{\partial x_i}-\sqrt{-1}\frac{\partial }{\partial y_i}\right),$$
and
\[
g_{i\bar{j}}=\frac{1}{2}\left(h_{ij}+\sqrt{-1}h_{i,m+j}\right),\ and\ g^{i\bar{j}}=2\left(h^{ij}-\sqrt{-1}h^{i,m+j}\right)
\]
where $z_i=x_i+\sqrt{-1}y_i$, $h_{ij}=h(\frac{\partial}{\partial x_i}, \frac{\partial}{\partial x_j})$, and $h_{i,m+j}=h(\frac{\partial}{\partial x_i}, \frac{\partial}{\partial y_j})$. Thus, for a real function $f$, we have
\[
\al|\d f|^2=&|\d_if|^2+|\d_{\g i} f|^2=2g^{i\bar{j}}\d_if\d_{\g j}f \\
=& h^{ij}\d_{\frac{\partial}{\partial x_i}}f\d_{\frac{\partial}{\partial x_j}}f+h^{m+i,m+j}\d_{\frac{\partial}{\partial y_i}}f\d_{\frac{\partial}{\partial y_j}}f+h^{i,m+j}\d_{\frac{\partial}{\partial x_i}}f\d_{\frac{\partial}{\partial y_j}}f +h^{m+i,j}\d_{\frac{\partial}{\partial y_i}}f\d_{\frac{\partial}{\partial x_j}}f,
\eal\]
and
\[
\al
\Delta f=&\frac{1}{2}g^{i\bar{j}}\left(\d_i\d_{\g j}+\d_{\g j}\d_i\right) f\\
=&\frac{1}{2}\left(h^{ij}\d_{\frac{\partial }{\partial x_i}}\d_{\frac{\partial }{\partial x_j}}+h^{m+i,m+j}\d_{\frac{\partial }{\partial y_i}}\d_{\frac{\partial }{\partial y_i}}+h^{i,m+j}\d_{\frac{\partial}{\partial x_i}}\d_{\frac{\partial}{\partial y_j}} +h^{m+i,j}\d_{\frac{\partial}{\partial y_i}}\d_{\frac{\partial}{\partial x_j}}\right)f.
\eal
\]
In other words, the norm of the complex and real gradients of $f$ are the same, and the complex Laplacian is half of the real Laplacian.

The next lemma presents a Bochner type formula involving only the Ricci curvature and the gradient of the scalar curvature. As mentioned, usually such results are expected to hold only for scalar functions. Also it does not follow from the Bochner formula for $(m, 1)$ forms where only the Ricci curvature appears. The reason is that the covariant derivative of an $(m, 0)$ form is not necessarily an $(m, 1)$ form. Later we will show that, in many situations, the gradient of the  scalar curvature can be integrated out, leaving the dependence only on Ricci curvature.

\begin{lemma} [Main Lemma]
\label{Bochner}
Let $\phi$ be a smooth $(m,0)$ form on a K\"ahler manifold $\M$ of complex dimension $m$. Then
\be
\lab{bochM}
\al
\Delta|\d \phi|^2=&|\d^2\phi|^2+<\d\Delta\phi, \d\phi>+<\d\phi,\d\Delta\phi>+R|\d\phi|^2+<\phi\d^{1,0}R, \d^{1,0}\phi>\\
&\ +<\d^{1,0}\phi, \phi\d^{1,0} R>+3Ric(\d^{1,0}\phi,\d^{1,0}\phi)-Ric(\d^{0,1}\phi,\d^{0,1}\phi)\\
 \equiv & |\d^2\phi|^2+<\d\Delta\phi, \d\phi>+<\d\phi,\d\Delta\phi>+R|\d\phi|^2+<\phi\d_jR,\d_{j}\phi>\\
&\  +<\d_j\phi, \phi\d_j R>+3<R_{j\bar{k}}\d_{k}\phi,\d_{j}\phi>-<R_{k\bar{j}}\d_{\bar{k}}\phi, \d_{\bar{j}}\phi>;
\eal
\ee
where $R=g^{i\bar{j}}R_{i\bar{j}}$ is the scalar curvature of $\M$, and $\d^{1,0}\phi=\d_i\phi_{12\cdots m}dz^1\wedge dz^2\wedge\cdots \wedge dz^m\otimes dz^i$ and $\d^{0,1}\phi=\d_{\g i}\phi_{12\cdots m}dz^1\wedge dz^2\wedge\cdots \wedge dz^m\otimes d\bar{z}^i$.

In particular, for K\"ahler-Einstein manifold $R_{i\g j}= \mu g_{i\g j}$, $\mu=\frac{R}{m}$, the Bochner formula becomes
\be\label{bochke}
\Delta|\d \phi|^2=|\d^2\phi|^2+<\d\Delta\phi, \d\phi>+<\d\phi,\d\Delta\phi>+m\mu|\d\phi|^2+3\mu|\d_{j}\phi|^2-\mu|\d_{\bar{j}}\phi|^2.
\ee
Furthermore, if $\M$ is a Calabi-Yau manifold, i.e., $\mu=0$, then the following curvature free identity holds:
\be
\lab{bochcy}
\Delta|\d \phi|^2=|\d^2\phi|^2+<\d\Delta\phi, \d\phi>+<\d\phi,\d\Delta\phi>.
\ee

\end{lemma}
\proof

We may compute at the center of a local normal coordinate system with
\[
\phi = \phi_{12\cdots m} dz^1 \wedge dz^2 \wedge \cdots \wedge dz^m.
\]To simplify the presentation, we will drop the $dz^i$ below.

First of all,
\be
\lab{12}
\al
\Delta|\d \phi|^2&=\d_i\d_{\bar{i}}(\d_j\phi_{12\cdots m}\overline{\d_{j}\phi_{12\cdots m}})+\d_i\d_{\bar{i}}(\d_{\bar{j}}\phi_{12\cdots m}\overline{\d_{\bar{j}}\phi_{12\cdots m}})\\
&\equiv I + II.
\eal
\ee
By direct computation, the first term on the RHS can be written as
\be
\lab{I=1}
\al
I=&\d_i\d_{\bar{i}}(\d_j\phi_{12\cdots m}\overline{\d_{j}\phi_{12\cdots m}})\\
=&\d_i\d_{\bar{i}}\d_j\phi_{12\cdots m}\overline{\d_{j}\phi_{12\cdots m}}+\d_j\phi_{12\cdots m}\overline{\d_{\bar{i}}\d_i\d_{j}\phi_{12\cdots m}}\\
& +\d_{\bar{i}}\d_j\phi_{12\cdots m}\overline{\d_{\bar{i}}\d_{j}\phi_{12\cdots m}}+\d_i\d_j\phi_{12\cdots m}\overline{\d_i\d_{j}\phi_{12\cdots m}}.
\eal
\ee Therefore
\be
\lab{I=2}
\al
I
=&\d_i\d_{\bar{i}}\d_j\phi_{12\cdots m}\overline{\d_{j}\phi_{12\cdots m}}+\d_j\phi_{12\cdots m}\overline{\d_{\bar{i}} \d_j \d_i\phi_{12\cdots m}}\\
& + |\d_{\bar{i}}\d_j\phi |^2+ |\d_i\d_j\phi|^2.
\eal
\ee

By \eqref{Ricci identity}, the following identity holds
\be\label{ji-ij}
\al
\d_j \d_{\g i} \phi - \d_{\g i} \d_j \phi
=& -R_{j\g i}\phi_{12\cdots m},
\eal
\ee
where $\phi_{i_1i_2\cdots i_m}$ with repeating subindices are zero by the definition of $(p,q)$ forms \eqref{def form}.
Also, it follows from \eqref{def form} and \eqref{Ricci identity} that
\be\label{ji-ij1}
\al
(\d_{\bar{i}}\d_j-\d_{j}\d_{\bar{i}})\d_i\phi=&-R_{\bar{i}ji\bar{l}}\d_l\phi_{12\cdots m}-\sum_kR_{\bar{i}jk\bar{l}}\d_i\phi_{1\cdots k-1 l k+1\cdots m}\\
=& -R_{\bar{i}ji\bar{l}}\d_l\phi_{12\cdots m}+R_{j\g i}\d_i\phi_{12\cdots m}.
\eal
\ee
Applying \eqref{ji-ij} and \eqref{ji-ij1} to the righthand side of (\ref{I=2}), we find
\be
\lab{I=}
\al
I=&\d_i\left(\d_j\d_{\bar{i}}\phi_{12\cdots m}+ R_{\g i j} \phi_{12\cdots m} \right) \overline{\d_{j}\phi_{12\cdots m}}\\ &+ \d_j\phi_{12\cdots m}\left(\overline{\d_{j}\d_{\bar{i}}\d_i\phi_{12\cdots m}-R_{\bar{i}ji\bar{l}}\d_l\phi_{12\cdots m}+R_{j\g i}\d_i\phi_{12\cdots m}}\right)\\
&+ |\d_i\d_{\bar{j}}\phi|^2+|\d_i\d_j\phi|^2.\\
\eal
\ee Recall that for $(m, 0)$ forms the complex Laplacian is $\Delta =\d_i \d_{\g i}$. Moreover, by the definition of the Ricci curvature \eqref{Ricci curvature} and Bianchi identity, one can see that
\[
R_{\bar{i}ji\bar{l}}=-R_{i\bar{i}j\bar{l}}=-R_{j\bar{l}};
\]
Therefore we can apply these properties and (\ref{ji-ij}) on the second term of the RHS of identity
\eqref{I=} to deduce:
\[
\al
I
=& \left(\d_j\Delta\phi_{12\cdots m}+\d_i(R_{\bar{i}j}\phi_{12\cdots m})\right)\overline{\d_{j}\phi_{12\cdots m}}\\
&+ \d_j\phi_{12\cdots m}\left(\overline{\d_{j}\Delta\phi_{12\cdots m}+\d_j(R \phi_{12\cdots  m})+2R_{j\bar{l}}\d_l\phi_{12\cdots m}}\right)\\
&+|\d_i\d_{\bar{j}}\phi|^2+|\d_i\d_j\phi|^2\\
=&\d_j\Delta\phi_{12\cdots m}\overline{\d_{j}\phi_{12\cdots m}}+\d_j\phi_{12\cdots m}\overline{\d_{j}\Delta\phi_{12\cdots m}} +(\d_jR) \, \phi_{12\cdots m}\overline{\d_j\phi_{12\cdots m}}\\
&+3R_{\bar{i}j}\d_i\phi_{12\cdots m}\overline{\d_{j}\phi_{12\cdots m}}+ \d_j\phi_{12\cdots m} \overline{\phi_{12\cdots m}} (\d_{\bar{j}}R) + R|\d_i \phi|^2+|\d_i\d_{\bar{j}}\phi|^2+|\d_i\d_j\phi|^2.
\eal
\] Here we have also used the identities
\be
\lab{3ident}
\d_i R_{\g i j} =\d_j R,  \quad \textrm{and}\quad
\overline{R_{\g i j}} = R_{\g j i}=R_{i\bar{j}}.
\ee

Similarly, the second term on the RHS of (\ref{12}) is
\[
\al
II=&\d_i\d_{\bar{i}}(\d_{\bar{j}}\phi_{12\cdots m}\overline{\d_{\bar{j}}\phi_{12\cdots m}})\\
=&\d_i\d_{\bar{i}}\d_{\bar{j}}\phi_{12\cdots m}\overline{\d_{\bar{j}}\phi_{12\cdots m}}+\d_{\bar{j}}\phi_{12\cdots m}\overline{\d_{\bar{i}}\d_i\d_{\bar{j}}\phi_{12\cdots m}}\\
& +\d_{\bar{i}}\d_{\bar{j}}\phi_{12\cdots m}\overline{\d_{\bar{i}}\d_{\bar{j}}\phi_{12\cdots m}}+\d_i\d_{\bar{j}}\phi_{12\cdots m}\overline{\d_i\d_{\bar{j}}\phi_{12\cdots m}}.
\eal
\]Thus
\[
\al
II=& \d_{\bar{j}}\Delta\phi_{12\cdots m}\overline{\d_{\bar{j}}\phi_{12\cdots m}}-R_{i\bar{j}\bar{i}l}\d_{\bar{l}}\phi_{12\cdots m}\overline{\d_{\bar{j}}\phi_{12\cdots m}}-\sum_kR_{i\bar{j}k\bar{l}}\d_{\bar{i}}\phi_{1\cdots k-1 l k+1\cdots m}\overline{\d_{\bar{j}}\phi_{12\cdots m}}\\
&+ \d_{\bar{j}}\phi_{12\cdots m}\left(\overline{\d_{\bar{j}}\d_{\bar{i}}\d_i\phi_{12\cdots m}-\d_{\bar{i}}(\sum_kR_{i\bar{j}k\bar{l}}\phi_{1\cdots k-1 l k+1\cdots m})}\right) +|\d_{\bar{i}}\d_{\bar{j}}\phi|^2+|\d_i\d_{\bar{j}}\phi|^2\\
=& \d_{\bar{j}}\Delta\phi_{12\cdots m}\overline{\d_{\bar{j}}\phi_{12\cdots m}}+\d_{\bar{j}}\phi_{12\cdots m}\left(\overline{\d_{\bar{j}}\Delta\phi_{12\cdots m}+\d_{\bar{j}}(R\phi_{12\cdots m})}\right)-\d_{\bar{j}}\phi_{12\cdots m} \overline{\phi_{12\cdots m}}\d_jR\\
&\ -R_{\bar{i}j}\d_{\bar{j}}\phi_{12\cdots m}\overline{\d_{\bar{i}}\phi_{12\cdots m}}+|\d_{\bar{i}}\d_{\bar{j}}\phi|^2+|\d_i\d_{\bar{j}}\phi|^2\\
=& \d_{\bar{j}}\Delta\phi_{12\cdots m}\overline{\d_{\bar{j}}\phi_{12\cdots m}}+\d_{\bar{j}}\phi_{12\cdots m}\overline{\d_{\bar{j}}\Delta\phi_{12\cdots m}} +R|\d_{\bar{i}} \phi|^2-R_{\bar{i}j}\d_{\bar{j}}\phi_{12\cdots m}\overline{\d_{\bar{i}}\phi_{12\cdots m}}\\
&\ +|\d_{\bar{i}}\d_{\bar{j}}\phi|^2+|\d_i\d_{\bar{j}}\phi|^2.
\eal
\] Note that $|\d \phi |^2 = |\d_i \phi|^2 + |\d_{\g i} \phi|^2$. Combining the computations above for I and II  with (\ref{12}) finishes the proof of the Lemma.\qed
\medskip

 Next we turn to gradient and eigenvalue estimates for the Hodge Laplacian on $(m, 0)$ forms and the proof of Theorem \ref{thmain} which is divided into a number of lemmas. The following is a flow chart of the proof: 1. mean value inequality for eigenforms; 2. mean value inequality for the gradient of eigenforms; 3.
bounds of sum of the first $k$ eigenvalue by a power of the last one via integration of the mean value inequality; 4. positive lower bound of $\lambda_1$ under a mild condition; 5. lower bound for all eigenvalues.

As a start, applying Moser's iteration, we will prove a mean value inequality  for $(m,0)$ eigenforms of the Hodge Laplacian. Only the scalar curvature is involved since it is well known that, formula (\ref{bochm0}) below i.e., the Bochner formula for $(m, 0)$ forms, involves only the scalar curvature. c.f. \cite{MK} Chapter 3.

\begin{lemma}\label{MVI}
Let $(\M, g_{i\bar{j}})$ be a compact K\"ahler manifold of complex dimension $m\geq 2$. Suppose that the scalar curvature $R$ of $\M$ satisfies $R \geq -mK$, for some $K\geq0$. Let $\lambda\geq0$ be an eigenvalue of the Hodge Laplacian on $(m,0)$ forms, and $\phi=\phi_{12\cdots m}dz^1\wedge\cdots\wedge dz^m$ an eigenform associated with $\lambda$, i.e.
\be
\lab{evf}
\Delta_d \phi = \lam \phi.
\ee Assume further that the following Sobolev inequality is satisfied,
\be\label{sobolev}
\left(\int_\M u^{\frac{2m}{m-1}}\right)^{\frac{m-1}{m}}\leq C_S\left(\int_\M |\d u|^2+ |\M|^{-\frac{1}{m}}\int_\M u^2\right)
\ee for all smooth functions on $\M$.
Then we have the mean value inequality
\be\label{estimate phi}
\al
\max_\M|\phi|^2\leq 4^{m(m-1)}C_S^{m}(\lambda+mK+|\M|^{-\frac{1}{m}})^m\int_\M|\phi|^2.
\eal
\ee
\end{lemma}
\proof

Recall that $\Delta=-\frac{1}{2}\Delta_d$ on functions. Hence, we have
\[
\al
-\Delta_d |\phi|^2 &= 2 \Delta |\phi|^2\\
& =
2 <\nabla_j \phi, \nabla_j \phi> + 2 <\nabla_{\bar j} \phi, \nabla_{\bar j} \phi>
+2 <\nabla_j \nabla_{\g j} \phi, \phi> +2 <\phi, \nabla_{\g j} \nabla_{j} \phi>.
\eal
\]For the last term on the right side, using \eqref{Ricci identity} for $(m, 0)$ forms yields the known identity:
\be
\lab{bochm0}
\al
-\Delta_d|\phi|^2(x)=2|\d\phi|^2(x)-<\Delta_d \phi,\phi>(x)-<\phi,\Delta_d\phi>(x)+2R(x)|\phi|^2(x).
\eal
\ee From (\ref{evf})
and Kato's inequality, we have
\be\label{Delta phi^2}
\al
\Delta|\phi|^2=-\frac{1}{2}\Delta_d|\phi|^2&=|\d\phi|^2-\lambda|\phi|^2+R|\phi|^2\\
&\geq \left|\d|\phi|\right|^2-(\lambda+mK)|\phi|^2.
\eal
\ee
Hence,
\[
\al
\int_\M 2|\phi|^{2p-2}\Delta|\phi|^2\geq \int_\M \left[|\phi|^{2p-2}\left|\d|\phi|\right|^2-(\lambda+mK)|\phi|^{2p}\right].
\eal
\]
Using integration by parts on the LHS, one gets that for $p\geq 1$,
\[
\int_\M -\frac{4(p-1)}{p^2}|\d|\phi|^p|^2\geq \int_\M \left[\frac{1}{p^2}\left|\d|\phi|^p\right|^2-(\lambda+mK)|\phi|^{2p}\right],
\]
i.e.,
\[
\int_\M |\d|\phi|^p|^2 \leq \frac{p^2(\lambda+mK)}{4p-3} \int_\M |\phi|^{2p}.
\]

Let $\alpha=\frac{m}{m-1}$, according to the Sobolev inequality \eqref{sobolev}, we have
\[
\al
\left(\int_\M |\phi|^{2p\alpha}\right)^{1/\alpha}&\leq C_S\left[\frac{p^2(\lambda+mK)}{4p-3}+|\M|^{-\frac{1}{m}}\right]\int_\M |\phi|^{2p}\\
&\leq p^2C_S(\lambda+mK+|\M|^{-\frac{1}{m}})\int_\M |\phi|^{2p}=\tilde{C}p^2\int_\M |\phi|^{2p},
\eal
\]
where $\tilde{C}=C_S(\lambda+mK+|\M|^{-\frac{1}{m}})$.
Setting $p=\alpha^{k-1}$, $k=1,2,3,\cdots$, it implies that
\[
\al
\left(\int_\M |\phi|^{2\alpha^k}\right)^{1/\alpha^k}&\leq \tilde{C}^{\alpha^{-(k-1)}}\alpha^{2(k-1)\alpha^{-(k-1)}}\left(\int_\M|\phi|^{2\alpha^{k-1}}\right)^{1/\alpha^{k-1}}\\
&\leq \tilde{C}^{\alpha^{-(k-1)}}4^{(k-1)\alpha^{-(k-1)}}\left(\int_\M|\phi|^{2\alpha^{k-1}}\right)^{1/\alpha^{k-1}}
\eal
\]
Therefore,
\[
\al
\max_\M|\phi|^2&\leq \tilde{C}^{\sum_{k=1}^{\infty}\alpha^{-(k-1)}}4^{\sum_{k=1}^{\infty}(k-1)\alpha^{-(k-1)}}\int_\M|\phi|^2
=4^{m(m-1)}\tilde{C}^{m}\int_\M|\phi|^2
\eal
\]
This finishes the proof. \qed\\

In particular, by normalizing $\phi$ so that $||\phi||_{L^2}=1$, Lemma \ref{MVI} implies that
\be\label{phi upper bound}
|\phi|^2\leq 4^{m(m-1)}C_S^{m}(\lambda+mK+|\M|^{-\frac{1}{m}})^{m}.
\ee

Next, from Lemma \ref{Bochner} and Moser iteration, we are able to prove the following gradient estimate for $(m,0)$ eigenforms of the Hodge Laplacian. This time the Ricci curvature is involved.

\begin{lemma}\label{lem gradient est}
Let $(\M, g_{i\bar{j}})$ and $\phi$ be the same as in Lemma \ref{MVI}. Assume that $|Ric|\leq K$. Then
\be\label{gradient estimate1}
\max_{\M} |\d\phi|^2\leq 4^{m(m+1)}m^m(m+1)^{m+1}C_S^m\left[\lambda+K+|\M|^{-\frac{1}{m}}\right]^m(\lambda+K)\int_{\M}|\phi|^2.
\ee

In particular, when $||\phi||_{L^2}=1$, we have
\be\label{gradient estimate}
|\d\phi|^2\leq 4^{m(m+1)}m^m(m+1)^{m+1}C_S^m\left[\lambda+K+|\M|^{-\frac{1}{m}}\right]^m(\lambda+K).
\ee
\end{lemma}

\proof

Let us first deal with the case where $K>0$. Let 
\[
v=|\d\phi|^2+A|\phi|^2
\]where $A >0$ is a constant to be chosen later. Once we can bound $v$ then the bound on $|\d \phi|$ follows. The reason for the appearance of $A |\phi|^2$ term is to help dealing with the term involving
 $\nabla R$ in the Bochner formula.
 By \eqref{Delta phi^2} and Lemma \ref{Bochner}, we have
\be\label{v}
\al
\Delta v\geq &-\left[\lambda+(m+4)K\right]|\d\phi|^2+\phi_{12\cdots m}\d_jR\overline{\d_j\phi_{12\cdots m}}+\overline{\phi_{12\cdots m}}\d_j\phi_{12\cdots m}\d_{\bar{j}}R\\
& +|\d^2\phi|^2+A|\d\phi|^2-(\lambda+mK)A|\phi|^2\\
\geq& -\left[\lambda+(m+4)K\right]v+\phi_{12\cdots m}\d_jR\overline{\d_j\phi_{12\cdots m}}+\overline{\phi_{12\cdots m}}\d_j\phi_{12\cdots m}\d_{\bar{j}}R\\
& +|\d^2\phi|^2.
\eal
\ee
For $p\geq 1$, multiple both sides above by $v^{2p-1}$ and then take integrals over $\M$. Since the second and third terms on the RHS of the inequality are conjugate complex functions, we get
\be\label{estimate 2re}
\al
&\int_\M \phi_{12\cdots m}\d_jR\overline{\d_j\phi_{12\cdots m}}v^{2p-1}+\overline{\phi_{12\cdots m}}\d_j\phi_{12\cdots m}\d_{\bar{j}}Rv^{2p-1}\\
=& 2 Re\left[\int_\M \phi_{12\cdots m}\d_jR\overline{\d_j\phi_{12\cdots m}}v^{2p-1}\right]\\
= & -2Re\int_\M \left(R \d_j\phi\overline{\d_j\phi}v^{2p-1}+R\phi\overline{\d_{\bar{j}}\d_j\phi}v^{2p-1}+ R\phi\overline{\d_j\phi}\d_jv^{2p-1}\right)\\
=& -2Re\int_\M\left(R|\d_j \phi|^2v^{2p-1}+(-\frac{\lambda}{2} R+R^2)|\phi|^2v^{2p-1}+(2p-1)v^{2p-2}R\phi\overline{\d_{j}\phi}\d_jv\right)\\
\geq & -2\int_{\M}\left[\left(mK+(\frac{\lambda mK}{2} +m^2K^2)A^{-1}\right)v^{2p} + \frac{2p-1}{p}mKA^{-1/2}v^p|\d v^p|\right]\\
\geq & -\int_{\M}\left(2mK+(\lambda mK + 4pm^2K^2)A^{-1}\right)v^{2p}-\int_{\M}\frac{2p-1}{2p^2}|\d v^p|^2,
\eal
\ee
where Cauchy-Schwarz inequality has been applied to the last step above. When going from the 4th to the 5th line in the preceding paragraph, we also used the inequality
\[
| \phi \overline{\d_{j}\phi}| \le A^{-1/2} v.
\]
Integrating \eqref{v} and using the last paragraph gives
\[
\al
\int_\M \frac{2p-1}{p^2}|\d v^p|^2=&-\int_\M v^{2p-1}\Delta v\\
\leq& \int_\M \left[(\lambda+5mK)+A^{-1}mK(\lambda+4pmK)\right]v^{2p}+\int_{\M}\frac{2p-1}{2p^2}|\d v^p|^2,
\eal
\]
i.e.,
\be\label{eq2.1}
\int_{\M} |\d v^p|^2\leq \int_{\M}2p^2\left[(\lambda+5mK)+A^{-1}mK(\lambda+4pmK)\right]v^{2p}.
\ee
Since the Ricci curvature is bounded, Sobolev inequality as in \eqref{sobolev} holds. Then combining \eqref{sobolev} and \eqref{eq2.1} gives
\[
\left(\int_{\M} v^{2\alpha p}\right)^{\frac{1}{\alpha}}\leq 2p^3 C_S \left[\lambda+5mK+A^{-1}mK(\lambda+4mK)+|\M|^{-\frac{1}{m}}\right]\int_{\M} v^{2p},
\]
where $\alpha=\frac{m}{m-1}$.
Letting $p=\alpha^{k}$, $k=0, 1,2,\cdots$, and performing Moser iteration as in Lemma \ref{MVI} yield
\[
\al
&\max_{\M} v^2\\
\leq& 2^m4^{m(m-1)}C_S^m\left[\lambda+5mK+A^{-1}mK(\lambda+4mK)+|\M|^{-\frac{1}{m}}\right]^m\int_\M v^2,
\eal
\]
which infers
\[
\max_{\M} v\leq 2^m4^{m(m-1)}C_S^m\left[\lambda+5mK+A^{-1}mK(\lambda+4mK)+|\M|^{-\frac{1}{m}}\right]^m\int_{\M} v.
\]
Since
\[
\al
\int_{\M} v=&\int_{\M} (|\d \phi|^2+A|\phi|^2)\\
=& \int_{\M} (-\d_{\g i}\d_i\phi\cdot\bar{\phi}-\d_i\d_{\g i}\phi\cdot\bar{\phi}+A|\phi|^2)\\
=& \int_{\M} (-2\Delta \phi\cdot\bar{\phi}-R|\phi|^2+A|\phi|^2)\\
\leq & (\lambda+mK+A)\int_{\M}|\phi|^2,
\eal
\]
gradient estimate \eqref{gradient estimate1} follows immediately by setting $A=K$.

The case where $K=0$, i.e., $\M$ is Calabi-Yau, is much easier. One just need to set $v=|\d \phi|^2$ and apply the Moser iteration directly to \eqref{bochcy} as in the proof of Lemma \ref{MVI}.
\qed \\

As a consequence of \eqref{phi upper bound}, \eqref{gradient estimate} in the preceding two Lemmas and Cauchy-Schwarz inequality, we get

\begin{corollary}\label{gradient estimate for combination}
Let $(\M, g_{i\bar{j}})$ be a compact K\"ahler manifold of complex dimension $m\geq 2$. Assume that $|Ric|\leq K$. Let $\phi_1$, $\phi_2$, $\cdots$, $\phi_k$ be orthonormal $(m,0)$ forms satisfying $\Delta_d\phi_j=\lambda_j\phi_j$, $j=1,2,\cdots, k$. Then for any sequence of real numbers $b_j$, $j=1,2,\cdots, k$, with $\sum_{j=1}^k b_j^2\leq 1$, the $(m,0)$ form $w=\sum_{j=1}^kb_j\phi_j$, the following estimate holds
\[
|\d w|^2+(\lambda_k+K)|w|^2\leq (2m)^{m+1}16^{2m(m-1)}C_S^m(\lambda_k+K+|\M|^{-\frac{1}{m}})^{m}(\lambda_k+K).
\]
\end{corollary}

\medskip
Using an argument as in the proof of Theorem 2.2 (2) in \cite{WZ}, we have
\begin{corollary}
Under the assumptions in Corollary \ref{gradient estimate for combination}, we have
\be\label{sum ineq lambda1}
\lambda_1+\lambda_2+\cdots+\lambda_k\leq 2(2m)^{m+1}16^{2m(m-1)}C_S^m\left[\lambda_k+K+|\M|^{-\frac{1}{m}}\right]^{m+1}|\M|.
\ee
\end{corollary}

\begin{proof}  Since the complex situation is somewhat different from the real case, we give a detailed proof.
For each $x\in\M$, since the complex rank of the $k$ vectors at $x$,  $(\d_{\g 1}\phi_1,\cdots,\d_{\g m}\phi_1)$, $(\d_{\g 1}\phi_2,\cdots,\d_{\g m}\phi_2),\cdots$, $(\d_{\g 1}\phi_k,\cdots,\d_{\g m}\phi_k)$ is no more than $m$, it is possible to find a unitary matrix $\left(a_{ij}\right)_{k\times k}$ such that $\psi_i=\sum_{j=1}^ka_{ij}\phi_j$, $1\leq i \leq k$ satisfy that $$\d_{\g l}\psi_i(x)=0,\ 1\leq l \leq m, \ m+1\leq i\leq k.$$ Here, for simplicity, we used $\d_{\g j}\phi_i$ to denote the coefficient $\d_{\g j}(\phi_i)_{1,2,\cdots,m}$ of $\d^{0,1}\phi_i$.

Then we derive from Corollary \ref{gradient estimate for combination} that
\[
\al
\sum_{j=1}^k\sum_{l=1}^m|\d_{\g l}\phi_j|^2(x)= & \sum_{i=1}^m\sum_{l=1}^m|\d_{\g l} \psi_i|^2(x)\\
\leq & m\max_i|\d \psi_i|^2(x)\\
\leq & (2m)^{m+1}16^{2m(m-1)}C_S^m(\lambda_k+K+|\M|^{-\frac{1}{m}})^{m}(\lambda_k+K).
\eal
\]
Thus integrating both sides gives \eqref{sum ineq lambda1} since
\[
2 \int_\M \sum_{l=1}^m|\d_{\g l}\phi_j|^2(x) =  \lambda_j,
\] which follows from the fact that $\Delta_d=-2\Delta$.
\end{proof}

As indicated in \cite{WZ}, it can be shown that \eqref{sum ineq lambda1} induces a lower bound of $\lambda_k$. Explicitly, first by induction one gets that

\begin{lemma}\label{algebraic lemma}
For $0\leq \lambda_1\leq \lambda_2\leq \cdots\leq \lambda_k\leq \cdots$, if the inequality
\be\label{sum ineq lambda}
\lambda_1+\lambda_2+\cdots+\lambda_k\leq c_0\lambda_k^{m+1}
\ee
holds for any $k\geq 1$, then one has
\be\label{LB lambda}
\lambda_k\geq c_1k^{\frac{1}{m}},
\ee
where $c_1=\min\{\lambda_1, \left(\frac{m}{c_0(m+1)}\right)^{\frac{1}{m}}\}$, and $m\geq 1$ is an integer.
\end{lemma}

\proof The conclusion follows from induction on $k$. Firstly, it is obvious that \eqref{LB lambda} holds for $k=1$.

Assume that \eqref{LB lambda} holds for all $1\leq i<k$. We show that \eqref{LB lambda} also holds for $k$. There will be two cases.  \\
Case 1: If $\lambda_k=0$, then $\lambda_1=0$. Hence $c_1=0$, and \eqref{LB lambda} is trivial for $k$.\\
Case 2: If $\lambda_k>0$, then we may argue by contradiction. Suppose that
\be\label{contra lambda}
\lambda_k < c_1k^{\frac{1}{m}}.
\ee
Dividing both sides of \eqref{sum ineq lambda} by $\lambda_k$, and using the induction hypothesis for $\lambda_1, \cdots, \lambda_{k-1}$, we get
\[
\al
\left(\frac{1}{k}\right)^{\frac{1}{m}}+\left(\frac{2}{k}\right)^{\frac{1}{m}}+\cdots +\left(\frac{k-1}{k}\right)^{\frac{1}{m}}+\left(\frac{k}{k}\right)^{\frac{1}{m}}\leq \frac{\lambda_1}{\lambda_k}+\frac{\lambda_2}{\lambda_k}+\cdots +\frac{\lambda_{k-1}}{\lambda_k}+1\leq c_0\lambda_k^{m}.
\eal
\]
Plugging in \eqref{contra lambda} yields
\be\label{eq3.1}
\left(\frac{1}{k}\right)^{\frac{1}{m}}+\left(\frac{2}{k}\right)^{\frac{1}{m}}+\cdots +\left(\frac{k-1}{k}\right)^{\frac{1}{m}}+\left(\frac{k}{k}\right)^{\frac{1}{m}}< \frac{m}{m+1}k.
\ee
However, by induction again, it is straightforward to check that
\be\label{eq3.2}
\left(\frac{1}{k}\right)^{\frac{1}{m}}+\left(\frac{2}{k}\right)^{\frac{1}{m}}+\cdots +\left(\frac{k-1}{k}\right)^{\frac{1}{m}}+\left(\frac{k}{k}\right)^{\frac{1}{m}}\geq \frac{m}{m+1}k,
\ee
for any $k\geq 1$, which contradicts with \eqref{eq3.1}. In fact, \eqref{eq3.2} obviously holds for $k=1$. If furthermore, it holds for $k-1$, i.e.,
\[
\left(\frac{1}{k-1}\right)^{\frac{1}{m}}+\left(\frac{2}{k-1}\right)^{\frac{1}{m}}+\cdots +\left(\frac{k-2}{k-1}\right)^{\frac{1}{m}}+\left(\frac{k-1}{k-1}\right)^{\frac{1}{m}}\geq \frac{m}{m+1}(k-1),
\]
then
\[
\al
\left(\frac{1}{k}\right)^{\frac{1}{m}}+\left(\frac{2}{k}\right)^{\frac{1}{m}}+\cdots +\left(\frac{k-1}{k}\right)^{\frac{1}{m}}+\left(\frac{k}{k}\right)^{\frac{1}{m}}\geq \frac{m}{m+1}(k-1)\left(\frac{k-1}{k}\right)^{\frac{1}{m}}+1.
\eal
\]
By direct computation, one can see that function $f(x)=\frac{m}{m+1}(1-x)^{\frac{1}{m}+1}+x-\frac{m}{m+1}$ has a global minimum at $x=0$. It implies that
\[
f(\frac{1}{k})\geq f(0)=0,
\]
which is equivalent to
\[
\frac{m}{m+1}(k-1)\left(\frac{k-1}{k}\right)^{\frac{1}{m}}+1\geq \frac{m}{m+1}k.
\]
Therefore, \eqref{eq3.2} also holds for $k$. This finishes the proof of the lemma.\qed\\

An immediate consequence of \eqref{sum ineq lambda1} and Lemma \ref{algebraic lemma} is that

\begin{corollary}\label{lambda k}
Under the assumptions in Corollary \ref{gradient estimate for combination}, if $\lambda_{k_0}\geq c_2>0$ for some $k_0$, then one has
\be\label{rough eigenvalue lower bound}
\lambda_k\geq c_3(k-k_0+1)^{\frac{1}{m}}, \forall\ k\geq k_0,
\ee
where $c_3=\min\{c_2, \left(\frac{m}{\Lambda(m+1)}\right)^{\frac{1}{m}}\}$, and $\Lambda=2(2m)^{m+1}16^{2m(m-1)}C_S^m|\M|\left(1+\frac{K+|\M|^{-\frac{1}{m}}}{c_2}\right)^{m+1}$.
\end{corollary}

On the other hand,  the following lower bound for the first eigenvalue $\lambda_1$ holds.

\begin{lemma}\label{lambda1 lower bound R>0} Let $\M$ be a compact K\"ahler manifold  of complex dimension $m$.
If the scalar curvature $R$ is non-negative and positive somewhere, then

\[
\lambda_1\geq \min\{\frac{1}{2} \lambda^0_1 + \inf_\M R, \quad \frac{\lambda^0_1}{2(\lambda^0_1 + \sup R)} \frac{1}{ |\M|} \int_\M R \}
\]where $\lambda^0_1$ is the first nontrivial eigenvalue of the scalar Laplacian.
\end{lemma}

\proof

Assume that $\phi$ is an eigenform of $\lambda_1$ with $\int_\M |\phi|^2=1$. Recall from  \eqref{Delta phi^2}
 that
\be
\lab{bochm00}
\Delta|\phi|^2=-\frac{1}{2}\Delta_d|\phi|^2=|\d\phi|^2-\lambda_1 |\phi|^2+R|\phi|^2.
\ee Integrating the above identity and using Kato's inequality, we find that
\be
\lab{lam11p}
\al
\lambda_1 &=\int_\M |\d\phi|^2 + \int_\M R|\phi|^2 \ge \int_\M |\d |\phi| |^2 + \int_\M R|\phi|^2\\
&\ge \lambda_{1, R},
\eal
\ee where $\lambda_{1, R}$ is the first eigenvalue of the scalar Schr\"odinger operator $-\Delta^{\mathbb{R}} +R$ with $\Delta^{\mathbb{R}} $ being the real Laplacian.

If $|\phi|$ is a constant function, then \eqref{lam11p} also gives us
\be
\lab{lam1phiconst}
\lambda_1
\ge \frac{1}{|\M|} \int_\M R.
\ee

 Next we assume that $|\phi|$ is not a constant.  For simplicity, we write $f=|\phi(x)|$ and $a = \frac{1}{|\M|} \int_\M f$, the average of $|\phi|$ over $\M$. We consider two cases.\\

\noindent {\it Case 1.} Suppose
\be
\lab{f-a>12}
\int_\M (f-a)^2 \ge 1/2.
\ee

Then the first line of \eqref{lam11p} implies
\[
\al
\lambda_1 & \ge \int_\M |\d f |^2 + \int_\M R f^2\\
&\ge \lambda^0_1 \int_{\M} (f-a)^2 + \inf_\M R.
\eal
\]
Therefore
\be
\lab{lam1>lam01/2}
\lambda_1 \ge  \frac{1}{2}\lambda^0_1 + \inf_\M R.
\ee

\noindent {\it Case 2.} Suppose
\be
\lab{f-a<12}
\int_\M (f-a)^2 < 1/2.
\ee

Then we can expand the square to reach:
\[
\int_{\M} f^2 - 2 a \int_{\M} f + a^2 |\M| <1/2,
\]which shows, since $\int_\M f^2 =1$,  that
\be
\lab{a2>}
a^2 > \frac{1}{2|\M|}.
\ee

From  the first line of \eqref{lam11p} again,

\[
\al
\lambda_1 & \ge \int_\M |\d f |^2 + \int_\M R (f-a +a)^2\\
&\ge  \int_\M \left(  \lambda^0_1 + R \right)  (f-a)^2 + 2 a \int_\M R (f-a) + a^2 \int_\M R \\
&=\int_\M \left(  \lambda^0_1 + R \right)  (f-a)^2 + 2 a \int_\M \frac{R}{\sqrt{ \lambda^0_1 +R}} \sqrt{ \lambda^0_1 +R} (f-a) + a^2 \int_\M R\\
&\ge - a^2 \int_\M \frac{  R^2}{\lambda^0_1 +  R}  + a^2 \int_\M R\\
&=a^2 \int_\M \frac{\lambda^0_1 R}{\lambda^0_1 +R}.
\eal
\] Here we just used Cauchy-Schwarz inequality in the 3rd last line.
From \eqref{a2>}, this shows
\be
\lambda_1 \ge a^2 \frac{\lambda^0_1}{\lambda^0_1 + \sup R} \int_\M R >  \frac{\lambda^0_1}{2(\lambda^0_1 + \sup R)} \frac{1}{ |\M|} \int_\M R.
\ee  Combining this with \eqref{lam1phiconst} and \eqref{lam1>lam01/2}, we find that
\[
\lambda_1 \ge \min\{\frac{1}{2} \lambda^0_1 + \inf_\M R, \quad \frac{\lambda^0_1}{2(\lambda^0_1 + \sup R)} \frac{1}{ |\M|} \int_\M R\}.
\]
This completes the proof of the lemma.


\qed

\begin{remark}
It is not hard to generalize from the above lemma that Hodge numbers $h^{m,0}=h^{0,m}=0$ whenever the total scalar curvature is positive and the negative part of the scalar curvature is sufficiently small in $L^\infty$ norm. Indeed in this case $\lambda_1>0$ and  hence there is no nonzero harmonic $(m,0)$ forms. The same conclusion for $R>0$ was first obtained by Kobayashi-Wu \cite{KW}.
\end{remark}

Now we are ready to give\\

\noindent{\it Completion of proof of Theorem \ref{thmain}.}

Theorem \ref{thmain} is a direct consequence of Corollary \ref{lambda k} with $k_0=1$ and Lemma \ref{lambda1 lower bound R>0}.\\  \qed\\

Another consequence of Lemma \ref{lambda1 lower bound R>0} is a more explicit lower bound of $\lambda_1$ for Fano manifolds. If $\M$ is a Fano manifold with
$\textbf{c}_1=\a[\omega]$ for some positive real number $\a$. Here $\textbf{c}_1$ is the first Chern class and $\omega$ is the K\"ahler form. By the $\partial\bar{\partial}$ lemma, there is a real smooth function $L$ such that $R_{i\bar{j}}=\pi\a g_{i\bar{j}}+\partial_i\partial_{\g j}L$. Thus, we get $R=m\pi\a+\Delta L$, and
\be\label{avg R}
\frac{1}{|\M|}\int_{\M}R=m\pi\a.
\ee
On the other hand, the relation between $\textbf{c}_1$ and $[\omega]$ also implies that
\[
\a^m|\M|=\a^m\int_{\M}[\omega]^m=\int_{\M}\textbf{c}_1^m.
\]
Since $\textbf{c}_1>0$, the RHS is a positive integer. Hence, it induces that
\[
\a^m|\M|\geq 1,
\]
which, from \eqref{avg R}, is equivalent to
\be\label{avg R1}
\frac{1}{|\M|}\int_{\M}R\geq m\pi|\M|^{-\frac{1}{m}}.
\ee
Therefore, substituting the bound above in Lemma \ref{lambda1 lower bound R>0}, we get

\begin{corollary}
Let $\M$ be a Fano manifold of complex dimension $m$. Suppose that $\mathbf{c}_1=\alpha[w]$ for some $\alpha>0$, and $0\leq R\leq K$ then

\[
\lambda_1\geq \min\{\frac{1}{2} \lambda^0_1, \quad \frac{\lambda^0_1m\pi}{2(\lambda^0_1 + K)} |\M|^{-\frac{1}{m}}\}.
\]

\end{corollary}


We end this section with a lower bound estimate for the first eigenvalue of 1 forms on some K\"ahler surfaces, which is similar in spirit to Theorem \ref{thmain}.

\begin{proposition}
 Suppose  $\M$ is a simply connected compact K\"ahler manifold of complex dimension 2 whose scalar curvature
is nonnegative and positive somewhere.  Let $\lambda^{(1)}_1$ be the
first eigenvalue of the Hodge Laplacian on $(1,0)$ forms. Then
\[
\lambda^{(1)}_1 \ge \min\{\frac{1}{2} \lambda^0_1, \quad \frac{\lambda^0_1}{2(\lambda^0_1 + \sup R)} \frac{1}{ |\M|}\int_{\M}R \}>0,
\]  where $\lambda^0_1$ is the first nonzero eigenvalue of the scalar Laplacian.

\end{proposition}
\proof
Let $\alpha$ be an eigenform for $\lambda^{(1)}_1$, i.e.
\[
\Delta_d \alpha = \lambda^{(1)}_1 \alpha.
\]Then applying $\partial$ on both sides yields
\[
\Delta_d \partial \alpha = \lambda^{(1)}_1 \partial \alpha.
\] If $\partial \alpha \neq 0$, then it is a $(2, 0)$ eigenform. By  Lemma \ref{lambda1 lower bound R>0}, the stated lower bound for
$\lambda^{(1)}_1$ is true. If $\partial \alpha = 0$, then we notice that $\partial^* \alpha $ can not be a constant function. Otherwise we would have
\[
\Delta_d \alpha = \partial \partial^* \alpha + \partial^* \partial \alpha =0.
\] Hence $\alpha$ is a nontrivial harmonic 1 form whose existence means that the first Betti number of
$\M$ is not $0$, contradicting with the assumption that $\M$ is simply connected.
Consequently $\partial^* \alpha$ is a nonconstant solution to the scalar equation:
\[
-2 \Delta \partial^* \alpha = \Delta_d \partial^* \alpha = \lambda^{(1)}_1 \partial^* \alpha.
\] Since when acting on real functions, $2\Delta=-\Delta_d$ is the real Laplacian, we have
\[
\lambda^{(1)}_1  \ge \lambda^{0}_1.
\] This completes the proof of the proposition. \qed

\begin{remark}
From the proof above, we can see that the ``simply connectedness" assumption can be dropped. Instead, $\lambda_1^{(1)}$ will be considered as the first nonzero eigenvalue.
\end{remark}


\section{The Heat kernel estimates}

In this section we prove pointwise and gradient bound for the heat kernel of the Hodge Laplacian on $(m, 0)$ forms. Again, the main feature of the results is that the bounds of the heat kernel and its gradient only rely on the Ricci curvature bound instead of the full curvature bound.

We consider time dependent, smooth $(m,0)$ forms $\phi$ which satisfy the heat equation
\be\label{heat eq}
(\partial_t-\Delta)\phi=(\partial_t+\frac{1}{2} \Delta_d) \phi =0
\ee on $\M \times (0, T]$.  Here, for consistency with the usual heat equation, we put $1/2$ in front of the Hodge Laplacian and
$\Delta$ is the complex Laplacian.
Let $\vec{G}=\vec{G}(x, t, y)$ be the heat kernel, i.e. fundamental solution of (\ref{heat eq}) such that
\[
\lim_{t \to 0^+} \int \vec{G}(x, t, y) \psi(y) dg(y) = \psi(x)
\]for all smooth $(m, 0)$ forms $\psi$.  We have

\begin{theorem}
\lab{thhk}
Let $(\M, g_{i\bar{j}})$ be a compact K\"ahler manifold of complex dimension $m$. Suppose that the Ricci curvature satisfies $|Ric|\leq K$. Then there exist a positive constant $A_1$ depending on the Sobolev constant $C_S$ in \eqref{sob} and
dimensional constants $a_2$ and $a_3$ such that
\be
\lab{hkjie}
\al
|\vec{G}(x, t, y) | &\le \frac{A_1}{t^{m}} e^{a_2 K t}  e^{- a_3 d^2(x, y)/t},\\
|\nabla_x \vec{G}(x, t, y) | &\le \frac{A_1}{t^{m+(1/2)}} e^{a_2 K t}  e^{- a_3 d^2(x, y)/t}\ when\  m\geq 2.
\eal
\ee
\end{theorem}

 In order to prove the theorem, we need two intermediate results.
The first result is a pointwise bound for $|\phi|$  via semigroup domination property.
Note we can also use Moser's iteration to get a similar result.

\begin{proposition}
Let $(\M^m,g_{i\bar{j}})$ be a compact K\"ahler manifold of complex dimension $m$ with $Ric \geq-K$, and $\phi=\phi(x,t)$ be a smooth $(m,0)$ form satisfying the heat equation \eqref{heat eq} on $[0,T]$. Then we have,
\be\label{global phi}
|\phi|^2(x, t) \leq \frac{A(m)}{t^m} e^{m K t}  \,  \Vert |\phi|(\cdot,0) \Vert^2_{L^2(\M)},\  t \in (0, T],
\ee
and \be\label{local phi}
\sup_{B_{\frac{\sqrt{T}}{2}}\times [\frac{3}{4}T, T]}|\phi|^2(x, t) \leq \frac{A(m)}{T^m} e^{m K T}  \,  \Vert |\phi|(\cdot,0) \Vert^2_{L^2(B_{\sqrt{T}})},
\ee
where $A(m)$ is a constant only depending on $m$ and $C_S$, the Sobolev constant in \eqref{sob}, and $B_r$ denotes the geodesic ball with radius $r$ centered at some point $O\in\M$.
\end{proposition}
\begin{proof}

Let $\phi$ be a solution to \eqref{heat eq}
Then
\[
\pa_t|\phi|^2=\Delta\phi\cdot\bar{\phi}+\phi\cdot\overline{\Delta\phi}=\Delta\phi\cdot\bar{\phi}+\phi\cdot\d_{\bar{i}}\d_i\bar{\phi},
\]
and
\[
\al
\Delta|\phi|^2&=\d_i\left(\d_{\bar{i}}\phi\cdot \bar{\phi}+\phi\cdot\d_{\bar{i}}\bar{\phi}\right)\\
&=\Delta\phi\cdot\bar{\phi}+\d_{\bar{i}}\phi\d_i\bar{\phi}+\d_i\phi\d_{\bar{i}}\bar{\phi}+\phi\d_i\d_{\bar{i}}\bar{\phi}\\
&=\Delta\phi\cdot\bar{\phi}+\phi\cdot\d_{\bar{i}}\d_i\bar{\phi}+R|\phi|^2+|\d\phi|^2.
\eal
\]

Therefore, $|\phi|$, the norm of $\phi$ satisfies the scalar equation:
\be\label{heat phi}
(\partial_t-\Delta)|\phi|^2=-R|\phi|^2-|\d\phi|^2.
\ee

From \eqref{heat phi} and the lower bound of the scalar curvature coming from $Ric$ lower bound assumption, we have
\[
(\partial_t-\Delta)|\phi|^2\leq m K|\phi|^2-|\d|\phi||^2.
\]
It implies that
\be
\lab{phisubso}
(\partial_t-\Delta)( e^{-m K t} |\phi|^{2} ) \le 0.
\ee Let $G=G(x, t, y)$ be the heat kernel of the standard scalar heat equation. The maximum principle infers that
\be
\lab{mktphi}
e^{-m K t} |\phi|^{2}(x, t) \le \int_\M G(x, t, y) |\phi(y, 0)|^2 dy.
\ee According to \cite{LY:2}, since $\M$ is compact, there exist a constant $C_1$ depending on $m$ and $|\M|$, and a dimensional positive constant $C_2$ such that
\be
\lab{ggub}
G(x, t, y) \le \frac{C_1}{|B(x, \sqrt{t})|} e^{-C_2 d^2(x, y)/t}.
\ee In fact $G(x, t, y)$ converges to $1/|\M|$ as $t \to \infty$. Since the Sobolev constant $C_S$ is a finite number, there exists another positive number, $A(m)$, depending on $C_S$, $|\M|$ and the dimension such that
\be
\lab{ggub2}
G(x, t, y) \le A(m) \left(\frac{1}{t^m} +1 \right) e^{-C_2 d^2(x, y)/t}.
\ee Substituting this to \eqref{mktphi} gives \eqref{global phi}.

For \eqref{local phi}, one just need a cut-off function $\psi(x,t)=\xi(d(x,O))\eta(t)$, where $\xi: \ \mathbb{R}\rightarrow [0,1]$ satisfies $\xi(u)=1$ in $[0, \frac{\sqrt{T}}{2}]$, $\xi(u)=0$ on $[\sqrt{T},\infty]$, $-\frac{8}{\sqrt{T}}\leq \xi'\leq 0$, $|\xi''|\leq \frac{8}{T}$, and $\eta:\ \mathbb{R}\rightarrow[0,1]$ satisfies $\eta(s)=1$ in $[\frac{3T}{4}, T]$, $\xi(u)=0$ on $[0,\frac{T}{2}]$, $0\leq \eta'(s)\leq \frac{8}{T}$. Then it is not hard to check that \eqref{local phi} follows by multiplying \eqref{heat phi} by $\psi^2$ and applying Duhamel's formula and the heat kernel bound \eqref{ggub2}.

\end{proof}

Further more, we can get similar gradient estimate for $\phi$, assuming two sided bound of the Ricci curvature. Here we need to integrate out the gradient of the scalar curvature appeared in the Bochner formula.

\begin{proposition}
Let $(\M^m,g_{i\bar{j}})$ be a compact K\"ahler manifold of complex dimension $m\geq 2$ with $|\Ric|\leq K$, $\phi=\phi(x,t)$ a smooth $(m,0)$ form satisfying the heat equation \eqref{heat eq} on $[0,T]$. Then we have

\[
\sup_{B_{\frac{\sqrt{T}}{4}}\times[\frac{15}{16}T,T]}|\d\phi|^2\leq \frac{A(m)e^{C(m)KT}}{T^{m+1}}||\phi(\cdot,0)||^2_{L^2(B_{\sqrt{T}/2})}.
\] Here $B_r$ denotes a geodesic ball of radius $r$ centered at some point $O\in\M$, $A(m)$ is a constant depending on $m$ and the Sobolev constant $C_S$ in \eqref{sobolev}, and $C(m)$ is just a dimensional constant.

\end{proposition}

\proof From the Bochner formula \eqref{bochM}, we have
\be\label{heat d phi}
\al
&(\partial_t-\Delta)|\d\phi|^2\\
=& -|\d^2\phi|^2-R|\d\phi|^2-\phi\d_jR\d_{\bar{j}}\bar{\phi}-\bar{\phi}\d_j\phi\d_{\bar{j}}R-3R_{j\bar{k}}\d_j\phi\d_{\bar{j}}\bar{\phi}+R_{j\bar{k}}\d_{\bar{j}}\phi\d_{k}\bar{\phi}.
\eal
\ee
Let $$v=|\d\phi|^2+A|\phi|^2$$
with $A$ to be determined.  Similar to the Lemma \ref{lem gradient est}, if $K=0$, we may simply set $A=0$. In the following we assume that $K>0$ and $A>0$. The reason for adding $A|\phi|^2$ in $v$ is that in \eqref{int by parts1} below, after integration by parts, the term $|\phi|^2v^{2p-1}$ comes up, then by the definition of $v$, this term can be controlled by $A^{-1}v^{2p}$ so that the power of $v$ becomes the same with the other terms, and makes it more convenient for iteration.

Combine \eqref{heat phi} and \eqref{heat d phi}, we have
\be\label{heat v2p}
\al
(\partial_t-\Delta)v^{p}\leq & -p|\d^2\phi|^2v^{p-1}+ p(m+4) Kv^{p}-p\phi\d_jR\d_{\bar{j}}\bar{\phi}v^{2p-1}-p\bar{\phi}\d_j\phi\d_{\bar{j}}Rv^{p-1}\\
& -p(p-1)v^{p-2}|\d v|^2.
\eal
\ee
For positive numbers $\delta<\sigma$, let $\eta(s)$ a cut-off function satisfying $\eta=0$ on $[0, T-\sigma r^2]$, $\eta=1$ on $[T-\delta r^2, T]$, and $0\leq \eta'(s)\leq \frac{2}{(\sigma-\delta)r^2}$. For positive numbers $\mu<\nu$, let $\xi(u)$ be a cut-off function such that $\xi=1$ on $[0,\mu r]$, $\xi=0$ on $[\nu r,\infty]$, and $-\frac{2}{(\nu-\mu)r}\leq\xi'\leq 0$.

Set $\psi=\xi(x)\eta(t)$. Multiplying both sides of \eqref{heat v2p} by $\psi^2v^p$ and taking integral over $\M\times[0,T]$ yield
\be\label{eq3.3}
\al
&\iint_{\M\times[0,T]} \psi^2v^p(\partial_t-\Delta)v^{p}\\
\leq& \iint_{\M\times[0,T]} p\Big[-\frac{1}{m}|\Delta\phi|^2+ (m+4)K-(\phi\d_jR\d_{\bar{j}}\bar{\phi}+\bar{\phi}\d_j\phi\d_{\bar{j}}R)\Big]\psi^2v^{2p-1}\\
& -\iint_{\M\times[0,T]}p(p-1)\psi^2v^{2p-2}|\d v|^2.
\eal
\ee
On the other hand
\be\label{eq3.4}
\al
\iint_{\M\times[0,T]} \psi^2v^p(\partial_t-\Delta)v^{p}=&\frac{1}{2}\int_\M\psi^2v^{2p}\Big|_T -\iint_{\M\times[0,T]}\psi v^{2p}\frac{\partial\psi}{\partial t}\\
& +\iint_{\M\times[0,T]}\Big(|\d(\psi v^p)|^2-v^{2p}|\d\psi|^2\Big).
\eal
\ee
Therefore, combining \eqref{eq3.3} and \eqref{eq3.4} shows that
\be\label{eq3.5}
\al
&\frac{1}{2}\int_\M\psi^2v^{2p}\Big|_T+\iint_{\M\times[0,T]}|\d(\psi v^p)|^2\\
\leq &\iint_{\M\times[0,T]} p\Big[-\frac{1}{m}|\Delta\phi|^2+ (m+4)K-(\phi\d_jR\d_{\bar{j}}\bar{\phi}+\bar{\phi}\d_j\phi\d_{\bar{j}}R)\Big]\psi^2v^{2p-1}\\
&-\iint_{\M\times[0,T]}p(p-1)\psi^2v^{2p-2}|\d v|^2+\iint_{\M\times[0,T]}\left(\psi\frac{\partial\psi}{\partial t}+|\d\psi|^2\right)v^{2p}
\eal
\ee
Similar to \eqref{estimate 2re}, using integration by parts and Cauchy-Schwarz inequality give
\be\label{int by parts1}
\al
&-\int_{\M}\left(\phi\d_jR\d_{\bar{j}}\bar{\phi}+\bar{\phi}\d_j\phi\d_{\bar{j}}R\right)\psi^2 v^{2p-1}\\
\leq & \int_\M \Big[ C(m)(K+pK^2A^{-1})\psi^2v^{2p}+\frac{1}{m}|\Delta\phi|^2\psi^2v^{2p-1}+\frac{1}{2}|\d(\psi v^p)|^2+\frac{1}{p}\psi^2v^{2p-2}|\d v|^2\Big].
\eal
\ee
Also,
\be\label{estimate psi}
\al
\psi\frac{\partial\psi}{\partial t}+|\d\psi|^2=\psi\xi\eta'+\eta^2|\d\xi|^2\leq 4\left(\frac{1}{(\sigma-\delta)r^2}+\frac{1}{(\nu-\mu)^2r^2}\right).
\eal
\ee
Thus, by plugging \eqref{int by parts1} and \eqref{estimate psi} into \eqref{eq3.5}, and setting $A=K$, it follows that
\be\label{int by parts}
\al
&\int_\M\psi^2v^{2p}\Big|_T+\iint_{\M\times[0,T]}|\d(\psi v^p)|^2\\
\leq & p^2C(m)\left(K+\frac{1}{(\sigma-\delta)r^2}+\frac{1}{(\nu-\mu)^2r^2}\right)\iint_{\M\times[0,T]}(\psi v^p)^2
:=p^2L\iint_{\M\times[0,T]}(\psi v^p)^2,
\eal
\ee
where $L=C(m)\left(K+\frac{1}{(\sigma-\delta)r^2}+\frac{1}{(\nu-\mu)^2r^2}\right)$.
By the definition of $\psi$, it is not hard to deduce from above that
\be\label{eq6}
\al
&\int_{B_{\mu r}}v^{2p}\vert_{T'}+\iint_{B_{\mu r}\times[T-\delta r^2, T]}|\d v^p|^2
\leq p^2L\iint_{B_{\nu r}\times[T-\sigma r^2,T]}v^{2p},
\eal
\ee
for any $T'\in[T-\delta r^2, T]$.

For $Ric\geq -K$, the following local Sobolev inequality holds (see e.g. \cite{Sa}).
\[
\left(\int_{B_r}|u|^{\frac{2m}{m-1}}\right)^{\frac{m-1}{m}}\leq C(r)\left[\int_{B_r}|\d u|^2+r^{-2}\int_{B_r}|u^2|\right],
\]
where $u\in C^{\infty}(\M)$, and $C(r)=e^{C(m)(1+\sqrt{K}r)}|B_r|^{-\frac{1}{m}}r^2$.
It implies that
\be\label{eq3'}
\iint_{B_{\mu r}\times[T-\delta r^2, T]}|\d v^p|^2\geq C(\mu r)^{-1}\int_{T-\delta r^2}^T\left(\int_{B_{\mu r}}v^{2p\a}\right)^{1/\a}ds-(\mu r)^{-2}\iint_{B_{\nu r}\times[T-\sigma r^2, T]}v^{2p}.
\ee
Moreover, from \eqref{eq6}, we can see that
\[
\al
&\left(\int_{T-\delta r^2}^T\left(\int_{B_{\mu r}}v^{2p\a}\right)^{1/\a}ds\right)\cdot\left(\iint_{B_{\nu r}\times[T-\sigma r^2, T]}v^{2p}\right)^{1/m}\\
\geq &(p^2L)^{-1/m} \left(\int_{T-\delta r^2}^T\left(\int_{B_{\mu r}}v^{2p\a}\right)^{1/\a}ds\right)\cdot\left(\sup_{[T-\delta r^2, T]}\int_{B_{\mu r}}v^{2p}\right)^{1/m}\\
\geq &(p^2L)^{-1/m} \int_{T-\delta r^2}^T\left(\int_{B_{\mu r}}v^{2p\a}\right)^{1/\a}\left(\int_{B_{\mu r}}v^{2p}\right)^{1/m}ds\\
\geq &(p^2L)^{-1/m}\left(\iint_{B_{{\mu r}}\times[T-\delta r^2, T]} v^{2p(1+\frac{1}{m})}\right).
\eal
\]
Multiplying both sides of \eqref{eq6} by $\left(\iint_{B_{\nu r}\times[T-\sigma r^2, T]}v^{2p}\right)^{1/m}$, and applying the estimates above, we derive
\be\label{iteration}
\al
&\iint_{B_{\mu r}\times[T-\delta r^2, T]} v^{2p(1+\frac{1}{m})}\\
\leq & C(\mu r)\left[p^2L+(\mu r)^{-2}\right](p^2L)^{\frac{1}{m}}\cdot\left(\iint_{B_{\nu r}\times[T-\sigma r^2,T]} v^{2p}\right)^{1+\frac{1}{m}}\\
\leq & p^3e^{C(m)(1+\sqrt{K}\mu r)}|B_{\mu r}|^{-\frac{1}{m}}\mu^{-2}\left(\mu^2r^2L+1\right)(L)^{\frac{1}{m}}\left(\iint_{B_{\nu r}\times[T-\sigma r^2,T]} v^{2p}\right)^{1+\frac{1}{m}}.
\eal
\ee
Now, suppose that $a\in[\frac{1}{2},1],\ \tau\in (0,1]$. Let $\beta=1+\frac{1}{m}$, and for $k=1,2,\cdots,$ set $p=\beta^{k-1}$, $\sigma=a^2+\frac{(a+\tau)^2-a^2}{2^{k-1}}$, $\delta=a^2+\frac{(a+\tau)^2-a^2}{2^{k}}$, $I_k=[T-(a^2+\frac{(a+\tau)^2-a^2}{2^{k}})r^2, T]$, $\mu=a+\frac{\tau}{2^k}$, $\nu=a+\frac{\tau}{2^{k-1}}$, and $r_k=(a+\frac{\tau}{2^k})r$, then \eqref{iteration} becomes
\[
\al
||v^2||_{\beta^k,B_{r_k}\times I_k}\leq & \beta^{3(k-1)\beta^{-k}}e^{C(m)(1+\sqrt{K}(a+2^{-k}\tau )r)\beta^{-k}}|B_{ ar}|^{-\frac{1}{m}\beta^{-k}}\\
&\cdot\left[C(m)2^{2k}e^{\sqrt{K}r}\tau^{-1}\right]^{(1+\frac{1}{m})\beta^{-k}}r^{-\frac{2}{m}\beta^{-k}}||v^2||_{\beta^{k-1},B_{r_{k-1}}\times I_{k-1}}.
\eal
\]
Using iteration on $k$, then letting $k\rightarrow 0$, we have
\[
\max_{B_{ar}\times[T-(ar)^2,T]}v^2\leq \frac{C_0(m)e^{C(m)\sqrt{K}r}}{r^2|B_{ar}|}\tau^{-(m+1)}||v^2||_{1,B_{(a+\tau)r}\times [T-[(a+\tau)r]^2,T]},
\]
where $C_0(m)=2^{8m(m+1)}m^{2(m+1)}$.

Next, we may use a method in \cite{LS} to reduce the $L^2$ mean value inequality above to $L^1$ mean value inequality. Let $r_j=(\sum_{k=0}^j2^{-k})r$, $Q_j=B_{r_k}\times [T-r_j^2,T]$ for $j=0,1,2,\cdots$. Then $Q_0=B_r\times[T-r^2, T]\subset Q_1\subset\cdots\subset Q_j\subset\cdots \subset B_{2r}\times[T-(2r^2), T]$, and
\[
\al
\sup_{Q_{j}}v^2\leq& A_0 2^{(j+1)(m+1)}\iint_{Q_{j+1}}v^2\leq A_0 2^{(j+1)(m+1)}\left(\sup_{Q_{j+1}}v^2\right)^{\frac{1}{2}}\iint_{B_{2r}\times[T-(2r)^2,T]}v,
\eal
\]
where $A_0=\frac{C(m)e^{C(m)\sqrt{K}r}}{r^2|B_{r}|}$. Denote by $\displaystyle \hat{A}=A_0 \iint_{B_{2r}\times[T-(2r)^2,T]}v$ and run iteration, we get
\[
\sup_{Q_0}v^2\leq \hat{A}^{\sum_{k=0}^j2^{-k}}2^{(m+1)\sum_{k=0}^j(k+1)2^{-k}}(\sup_{Q_j}v^2)^{2^{-j}}.
\]
Letting $j\rightarrow \infty$, one can see that
\[
\sup_{Q_0}v^2\leq C(m)\hat{A}^2,
\]
i.e.,
\be\label{eq7}
\sup_{B_r\times [T-r^2,T]}v\leq \frac{C(m)e^{C(m)\sqrt{K}r}}{r^2|B_{r}|}\iint_{B_{2r}\times[T-(2r)^2,T]}v.
\ee

At last, let us bound the right hand side of the previous inequality.
Multiplying \eqref{heat phi} by the spatial cutoff function $\xi^2$ and doing the standard energy estimate, taking $r=\frac{1}{4}\sqrt{T}$, we can deduce with the assistance of \eqref{local phi} that
\be
\label{eq9}
\al
\iint_{B_{\sqrt{T}/2}\times[3T/4,T]}&\left( |\d\phi|^2 + |\phi|^2 \right) \xi^2
\leq & C(m)e^{mKT}||\phi(x,0)||^2_{L^2(B_{\sqrt{T}})}.
\eal
\ee
Now the conclusion of the proposition follows from \eqref{eq7}, \eqref{eq9} and the volume comparison.
\qed\\

Now we are ready to give a
\medskip

\noindent {\it  Proof of Theorem \ref{thhk}.}

{\it Step 1.}  The pointwise bound for $\vec{G}$ follows quickly from \eqref{phisubso}. Fixing $y$,
let $\phi(x, t) = \vec{G}(x, t, y)$. Then, \eqref{phisubso} infers that
\[
(\partial_t-\Delta)( e^{-m K t} |\phi(x, t)| ) \le 0.
\]Hence $e^{-m K t} |\phi(x, t)|$ is dominated by $G(x, t, y)$, the heat kernel of the scalar Laplacian.
So the pointwise bound in \eqref{hkjie} for $|\vec{G}|$ follows from \eqref{ggub2}.
\medskip

{\it Step 2 .} We prove the gradient bound.
\be
\lab{hkjie0}
|\nabla_x \vec{G}(x, t, y) | \le \frac{A_1}{t^{m+(1/2)}} e^{a_2 K t}.
\ee  where $A_1$ depends only on $K$, $m$ and $C_S$ and $a_2$ depends only on $m$.

Let $\phi$ be as in Step 1.  We apply Proposition 3.3 on the region $B_{\sqrt{t}/2}(x) \times [t/2, t]$ with $t/2$ taken as the initial time. Note that we are free to adjust the total time interval by a fixed factor. This gives
\be\label{MV heat kernel}
|\nabla \phi(x, t) |^2 \le \frac{A_0 e^{a_2Kt}}{t^{m+1}}  \int_{B_{\sqrt{t}/2}(x)} |\vec{G}(z, t/2, y)|^2 dz.
\ee
Here $A_0$ depends only on $K$, $m$ and $C_S$. From the pointwise bound proven in Step 1, after routine computation,  this implies
\[
|\nabla \phi(x, t) |^2 \le \frac{A_1 e^{a_2Kt}}{t^{2 m+1}}
\]proving \eqref{hkjie0}.
\medskip

{\it Step 3 .} Completion of the proof of the gradient bound in \eqref{hkjie}.

Since the manifold is compact, the exponential term $e^{- a_3 d^2(x, y)/t}$ is mute for $t \ge 1$.  So, by Step 2,  we only need to deal with the case when $t \in (0, 1]$ and $d^2(x, y) \ge 4 t$.

Now from \eqref{MV heat kernel} and using the bound on $|\vec{G}|$ and the property that $d(z, y) \ge d(x, y)/2$ for $z \in {B_{\sqrt{t}/2}(x)}$,
we deduce
\be
\al
|\nabla \phi(x, t) |^2
&\le \frac{A_0 e^{a_2 Kt}}{t^{m+1}}  \int_{B_{\sqrt{t}/2}(x)} \frac{A^2_1}{t^{2 m}} e^{ a_2 K t}  e^{- 4 a_3 d^2(z, y)/t} dz\\
&\le \frac{A_0 e^{a_2 Kt}}{t^{m+1}}  \Big|B_{\sqrt{t}/2}(x)\Big| \frac{A^2_1}{t^{2 m}} e^{ a_2 K t}  e^{-  a_3 d^2(x, y)/t}.
\eal
\ee  The desired gradient bound then follows by volume comparison, after a suitable adjustment of constants.
\qed

\section*{Acknowledgements}

Z. L. is supported by NSF grant DMS-19-08513. Q. S. Z. is supported by the Simons foundation grant 710364. M. Z. is supported by Science and Technology Commission of Shanghai Municipality (STCSM) No. 18dz2271000.



\begin{thebibliography}{00}

\bibitem[Bo]{Bo} Bochner, S., {\it Vector fields and Ricci curvature}, Bull. Amer. Math. Soc. 52 (1946), 776-797.


\bibitem[Cro]{Cro} Croke, C., {\it Some isoperimetric inequalities and eigenvalue estimates}, Ann. Sci. \'Ecole Norm. Sup. 13 (1980), no. 4, 419-435.

\bibitem[CL]{CL} Charalambous, N.; Lu, Z., {\it The spectrum of continuously perturbed operators and the Laplacian on forms}. Differential Geom. Appl. 65 (2019), 227-240.

\bibitem[CT]{CT} Chanillo, S.; Treves, F., {\it
On the lowest eigenvalue of the Hodge Laplacian.}
J. Differential Geom. 45 (1997), no. 2, 273-287.

\bibitem[Do]{Do} Dodziuk, J., {\it Eigenvalues of the Laplacian on Forms,} Proc. Amer. Math. Soc. 85 (1982), no. 3, 437-443.

\bibitem[Ko]{Ko} Kobayashi, S.,
{\it On compact K\"{a}hler manifolds with positive definite Ricci tensor.}
Ann. of Math. (2) 74 (1961), 570-574.

\bibitem[KW]{KW} Kobayashi, S.; Wu, H.-H., {\it On holomorphic sections of certain hermitian vector bundles.}
Math. Ann. 189 (1970), 1-4.


\bibitem[Ma]{Ma} Mantuano, T., {\it Discretization of Riemannian manifolds applied to the Hodge Laplacian.}  Amer. J. Math. 130 (2008), no. 6, 1477-508.

\bibitem[Li1]{Li:1} Li, P., {\it On the Sobolev constant and the p-spectrum of a compact Riemannian manifold,}
     Ann. Sci. \'Ecole Norm. Sup. 13 (1980), no. 4, 451-468.

\bibitem[Li2]{Li} Li, P., {\it Geometric analysis}, Cambridge Studies in Advanced Mathematics, 134. Cambridge University Press, 2012.

\bibitem[Lo]{Lo} Lott, J., {\it  Collapsing and the differential form Laplacian: the case of a smooth limit space}. Duke Math. J. 114 (2002), no. 2, 267-306.

\bibitem[LS]{LS} Li, P.; Schoen, R., {\it $L^p$ and mean value properties of
subharmonic functions on Riemannian manifolds}, Acta Math. 153 (1984), no. 3-4, 279-301.

\bibitem[LY]{LY} Li, P.; Yau, S.-T., {\it Eigenvalues of a compact Riemannian manifold,} AMS Proc. Symp.Pure Math. 36 (1980), 205-239.


\bibitem [LY2]{LY:2} Li, P.; Yau, S.-T, {\it On the parabolic kernel of the Schr\"odinger operator.}  Acta Math. 156 (1986), no. 3-4, 153-201.

\bibitem[MMZ]{MMZ} Ma, X.; Marinescu, G.; Zelditch, S., {\it
Scaling asymptotics of heat kernels of line bundles}.  Analysis, complex geometry, and mathematical physics: in honor of Duong H. Phong, 175-202, Contemp. Math., 644, Amer. Math. Soc., Providence, RI, 2015.

\bibitem[MK]{MK} Morrow, J.; Kodaira, K., {\it Complex manifolds.}  Holt, Rinchart and Winston, Inc., New York-Montreal-London 1971.



\bibitem[Sa]{Sa} Saloff-Coste, L., {\it  A note on Poincar\'e, Sobolev, and Harnack inequalities},
International Mathematics Research Notices 1992 (1992), no. 2, Pages 27-38.

\bibitem[WZ]{WZ} Wang, J. P. ; Zhou, L. F., {\it Gradient estimate for eigenforms of hodge laplacian}, Math. Res. Lett. 19 (2012), no. 3, 575-588.

\bibitem[Ti]{Ti} Tian, G., {\it On Calabi's conjecture for complex surfaces with positive first Chern class,}
Inventiones mathematicae 101 (1990), pages 101-172

\bibitem[Ya]{Ya}  Yang, H.-C., {\it Estimates of the first eigenvalue for a compact Riemann manifold}, Sci. China Ser. A 33 (1990), no. 1, 39-51.

\bibitem[Yau]{Yau} Yau, S.-T., {\it Isoperimetric constants and the first eigenvalue of a compact Riemannian manifold}, Ann. Sci. \'Ecole Norm. Sup. 8 (1975), no. 4, 487-507.



\bibitem[ZY]{ZY} Zhong, J.-Q.; Yang, H.-C., {\it On the estimate of the first eigenvalue of a compact Riemannian manifold.} Sci. Sinica Ser. A 27 (1984), no. 12, 1265-1273.

\end{thebibliography}
\end{document}